\documentclass[11pt]{amsart}
\usepackage{psfrag}
\usepackage[dvips]{graphicx}
\usepackage{amsmath}
\usepackage{hyperref} 

\begin{document}
%
%
\newtheorem{theorem}{Theorem}
\newtheorem{algorithm}[theorem]{Algorithm}
\newtheorem{conjecture}[theorem]{Conjecture}
\newtheorem{axiom}[theorem]{Axiom}
\newtheorem{corollary}[theorem]{Corollary}
\newtheorem{definition}[theorem]{Definition}
\newtheorem{example}[theorem]{Example}
\newtheorem{question}[theorem]{Question}
\newtheorem{lemma}[theorem]{Lemma}
\newtheorem*{case1}{Case 1}
\newtheorem*{case2}{Case 2}
\newtheorem{fact}[theorem]{Fact}
\newtheorem{proposition}[theorem]{Proposition}
\newcounter{case_counter}\setcounter{case_counter}{1}
\newtheorem{claim}{Claim}[case_counter]
\newtheorem*{claim1}{Claim 1}
\newtheorem*{claim2}{Claim 2}

%
%
\newcommand{\OMIT}[1]{}
\def\N{\mathbb{N}}
\def\Z{\mathbb{Z}}
\def\R{\mathbb{R}}
\def\P{\mathcal{P}}
\def\Q{\mathcal{Q}}
\def\C{\mathcal{C}}
\def\U{\mathcal{U}}
\def\E{\mathbb{E}}
\def\F{\mathcal{F}}
\def\eul{\rm{e}}
\def\e{\rm{e}}
\renewcommand{\Gamma}{\varGamma}
\renewcommand{\epsilon}{\varepsilon}
\newcommand{\eps}{\varepsilon}
\newcommand{\floor}[1]{\lfloor{#1}\rfloor}
\newcommand{\ceil}[1]{\lceil{#1}\rceil}
\renewcommand{\bar}{\overline}
\renewcommand{\hat}{\widehat}
\newcommand{\ie}{i.e.\ }
\newcommand{\etc}{etc.\ }
\newcommand{\wrt}{w.r.t.\ }
\newcommand{\whp}{w.h.p.\ }
\newcommand{\sm}{\setminus}
\newcommand{\nib}[1]{\noindent {\bf #1}}
\newcommand{\cproof}[2]{\nib{#1} #2 \qed\medskip}

%
%
\def\COMMENT#1{}
\def\TASK#1{}
\def\noproof{{\unskip\nobreak\hfill\penalty50\hskip2em\hbox{}\nobreak\hfill%
       $\square$\parfillskip=0pt\finalhyphendemerits=0\par}\goodbreak}
\def\endproof{\noproof\bigskip}
\newdimen\margin   
\def\textno#1&#2\par{%
   \margin=\hsize
   \advance\margin by -4\parindent
          \setbox1=\hbox{\sl#1}%
   \ifdim\wd1 < \margin
      $$\box1\eqno#2$$%
   \else
      \bigbreak
      \hbox to \hsize{\indent$\vcenter{\advance\hsize by -3\parindent
      \sl\noindent#1}\hfil#2$}%
      \bigbreak
   \fi}
\def\proof{\removelastskip\penalty55\medskip\noindent{\bf Proof. }}
\def\enddiscard{}
\long\def\discard#1\enddiscard{}

%
%
\title{Arbitrary Orientations Of Hamilton Cycles In Oriented Graphs}
\author{Luke Kelly}
\begin{abstract}
We use a randomised embedding method to prove that for
all~$\alpha>0$ any sufficiently large oriented graph~$G$ with
minimum in-degree and out-degree~$\delta^+(G),\delta^-(G)\geq
(3/8+\alpha)|G|$ contains every possible orientation of a
Hamilton cycle. This confirms a conjecture of H\"aggkvist and
Thomason.
\end{abstract}
\maketitle

\section{Introduction}\label{sec:intro}

An \emph{oriented graph} is a loop-free simple graph where each
edge is given an orientation. A \emph{directed graph (digraph)}
is an oriented graph where we allow one edge in each direction
between each pair of vertices, that is, we allow cycles of
length 2. The \emph{minimum semi-degree~$\delta^0(G)$} of an
oriented graph~$G$ (or a digraph) is the minimum of its minimum
outdegree~$\delta^+(G)$ and its minimum indegree~$\delta^-(G)$.

A fundamental result of Dirac states that a minimum degree
of~$|G|/2$ guarantees a Hamilton cycle in any undirected
graph~$G$ on at least~3 vertices. Following this result several
weaker conditions guaranteeing a Hamilton cycle have been
found. One of the famous of these is Ore's theorem, which
states that if~$d(x)+d(y)\geq |G| \geq 3$ for all $x\neq y\in
V(G)$ with $xy\not\in E(G)$ then~$G$ contains a Hamilton cycle.
In some sense the weakest possible condition of this type is
Chv\'atal's theorem.\footnote{Whilst it is widely regarded as
such, it should be noted that Chv\'atal's theorem does not
quite imply Ore's theorem.} This gives a condition on the
(ordered) degree sequence of a graph which forces a Hamilton
cycle, such that for any (graphic) degree sequence not
satisfying Chv\'atal's conditions there exists a graph with a
degree sequence dominated by that sequence not containing a
Hamilton cycle.

There is an analogue of Dirac's theorem for digraphs due to
Ghouila-Houri~\cite{GhouilaHouri} which states that every
digraph~$D$ with minimum semi-degree at least~$|D|/2$ contains
a  directed Hamilton cycle. As with Dirac's theorem, taking two
disjoint cliques of as equal size as possible demonstrates that
this minimum degree condition can not be improved.

Thomassen~\cite{thomassen_79} asked the natural question of
whether there exists an analogous result for oriented graphs,
where one expects to be able to obtain a weaker degree
condition than the bounds needed for digraphs.
H\"aggkvist~\cite{HaggkvistHamilton} constructed an example in
1993 showing that a minimum semi-degree of~$(3n-4)/8$ was
necessary and conjectured that this was also sufficient. With
Thomason~\cite{HaggkvistThomasonHamilton} he showed in 1997
that for any~$\alpha>0$ every sufficiently large oriented
graph~$G$ with minimum semi-degree at least $(5/12 + \alpha)
|G|$ has a directed Hamilton cycle. The author, together with
K\"uhn and Osthus~\cite{kelly_kuhn_osthus_hc_orient}, finally
confirmed in 2008 that, up to a linear error term,~$3|G|/8$ is
indeed the correct bound. Following this Keevash, K\"uhn and
Osthus improved this to an exact result.

\begin{theorem}[Keevash, K\"uhn and Osthus
\cite{KKO_exact}]\label{thm:ham_exact} There exists~$n_0$ such
that every oriented graph~$G$ on $n\geq n_0$ vertices with
$\delta^0(G)\geq (3n-4)/8$ contains a directed Hamilton cycle.
\end{theorem}

Christofides, Keevash, K\"uhn and Osthus \cite{CKKO} have also
since found an efficient algorithmic proof of (a generalisation
of) this result.

Nash-Williams~\cite{NashWilliams} conjectured a digraph
analogue of Chv\'atal's theorem. This has recently been
approximately confirmed by K\"uhn, Osthus and Treglown
\cite{KOT_hc}. There also now exists a semi-exact degree
condition result due to Christofides, Keevash, K\"uhn and
Osthus \cite{CKKO_semiexact}.

 It is natural to ask whether these bounds
only give us directed Hamilton cycles or whether they give
every possible orientation of a Hamilton cycle. Indeed this
question was answered for digraphs, asymptotically at least, by
H\"aggkvist and Thomason in 1995.

\begin{theorem}[H\"aggkvist and Thomason \cite{HT_arb}]
There exists~$n_0$ such that every digraph~$D$ on~$n\geq n_0$
vertices with minimum semi-degree $\delta^0(D)\geq n/2+n^{5/6}$
contains every orientation of a Hamilton cycle.
\end{theorem}

The question was asked originally for oriented graphs by
H\"aggkvist and Thomason~\cite{HaggkvistThomasonHamilton} who
proved that for all~$\alpha>0$ and all sufficiently large
oriented graphs~$G$ a minimum semi-degree of~$(5/12+\alpha)|G|$
suffices to give \emph{any} orientation of a Hamilton cycle.
They conjectured that~$(3/8+\alpha)|G|$ suffices, the same
bound as for the directed Hamilton cycle up to the error
term~$\alpha|G|$. Whilst not asked explicitly before
H\"aggkvist and Thomason's paper, there is some previous work
of Thomason and Grant relevant to this area. Grant~\cite{Grant}
proved in 1980 that any digraph~$D$ with minimum
semi-degree~$\delta^0(D)\geq 2|D|/3+\sqrt{|D|\log |D|}$
contains an anti-directed Hamilton cycle, provided that~$n$ is
even. (An anti-directed cycle is one in which the edge
orientations alternate.) Thomason~\cite{Thomason} showed in
1986 that every sufficiently large tournament contains every
possible orientation of a Hamilton cycle (except possibly the
directed Hamilton cycle if the tournament is not strong). The
following theorem confirms the conjecture of H\"aggkvist and
Thomason.

\begin{theorem}\label{thm:arb_ham}
For every~$\alpha>0$ there exists an integer $n_0=n_0(\alpha)$
such that every oriented graph~$G$ on~$n\geq n_0$ vertices with
minimum semi-degree~$\delta^0(G)\geq (3/8+\alpha)n$ contains
every orientation of a Hamilton cycle.
\end{theorem}

\subsection{Robust Expansion}

The property underlying the proofs of all the recent Hamilton
cycle results so far stated is \emph{robust expansion}. This is
a notion which was introduced by K\"uhn, Osthus and Treglown in
\cite{KOT_hc} and has proved to be the correct notion of
expansion in a digraph when dealing with this kind of question
or when using the Diregularity lemma. Informally speaking, a
digraph~$G$ is a robust outexpander if all subsets of~$V(G)$
have outneighbourhoods larger than themselves unless they are
very large or very small and, moreover, this still holds after
the removal of a small number of edges.

Having a minimum semi-degree~$\delta^0(G)\geq (3/8+\alpha)|G|$
for some~$\alpha>0$, satisfying an approximate Ore-type
condition or satisfying an approximate Chv\'atal condition
imply robust outexpansion (see Lemma~11 in~\cite{KOT_hc}).
Hence an extension of Theorem~\ref{thm:arb_ham} to robust
outexpanders would imply an approximate Ore-type result and a
Chv\'atal-type approximate result for arbitrary orientations of
Hamilton cycles. The author believes it is likely that the
argument given in this paper could be straight-forwardly
extended to prove this.

\subsection{Extremal Example}
H\"aggkvist~\cite{HaggkvistHamilton} constructed an example in
1993 giving a graph on~$n=8k-1$ vertices with minimum
semi-degree~$(3n-5)/8$ containing no Hamilton cycle and
Keevash, K\"uhn and Osthus extended this to all~$n$. This means
that Theorem~\ref{thm:ham_exact} is best possible and that
Theorem~\ref{thm:arb_ham} is best possible up to the linear
error term. Interestingly, this example can be improved upon
when considering arbitrary orientations. Hence the additive
constant in Theorem~\ref{thm:ham_exact} is not the correct
bound when seeking any orientation of a Hamilton cycle, and it
is an open question as to what the correct additional term
should be.

\begin{figure}
  \centering\footnotesize
  \centering\footnotesize
  \psfrag{1}[][]{$A$}
  \psfrag{2}[][]{$B$}
  \psfrag{3}[][]{$C$}
  \psfrag{4}[][]{$D$}
  \includegraphics[scale=1]{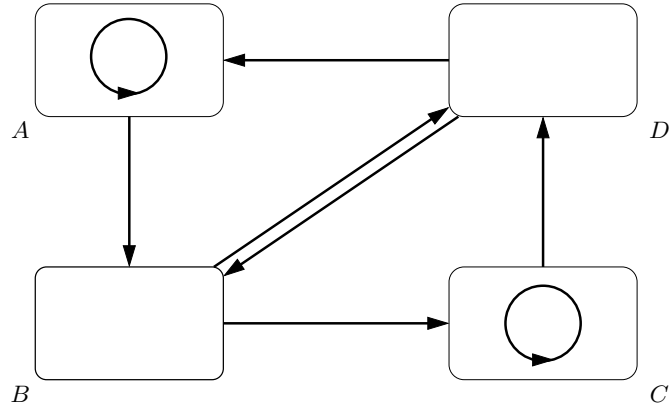}
  \caption{The oriented graph constructed in Proposition~\ref{prop:extremal}}\label{fig:extremal}
\end{figure}

\begin{proposition}\label{prop:extremal}
There are infinitely many oriented graphs~$G$ with minimum
semi-degree exactly~$(3|G|-4)/8$ which do not contain an
anti-directed Hamilton cycle.
\end{proposition}
\begin{proof}
Let $n:=8m+4$ for some integer $m\in\N$. Let~$G$ be the
oriented graph obtained from the disjoint union of two regular
tournaments~$A$ and~$C$ on~$2m+1$ vertices and sets~$B$ and~$D$
of~$2m+1$ vertices by adding all edges from~$A$ to~$B$, all
edges from~$B$ to~$C$, all edges from~$C$ to~$D$ and all edges
from~$D$ to~$A$. Finally, between~$B$ and~$D$ we add edges to
obtain a bipartite tournament which is as regular as possible,
i.e.~the indegree and the outdegree of every vertex differs by
at most~1. So in particular every vertex in~$B$ sends at least
$m$ edges to~$D$. It is easy to check that the minimum
semi-degree of~$G$ is $3m+1=(3n-4)/8$, as required.

Let us try to construct an anti-directed Hamilton cycle in~$G$
and let us start in~$B$ with an edge going forwards. This edge
can go either to~$C$ or to~$D$. (Starting with an edge oriented
backwards produces an identical argument and result.) The next
edge must go backwards. It can go from~$C$ to either~$B$
or~$C$. It can go from~$D$ to either~$B$ or~$C$. So after two
steps we can be in either~$B$ or $C$. Our next edge must go
forwards. If we are in~$B$ our possible locations after the
next two steps are~$B$ and~$C$ as before. From~$C$ we can go
forwards either to~$C$ or to~$D$. Both options repeat
situations we have already met. In no case do we have a means
to reach~$A$ whilst respecting the orientation of our
anti-directed Hamilton cycle. Hence the longest anti-directed
cycle in~$G$ has length at most~$3n/4$ and we have no
anti-directed Hamilton cycle as claimed.
\end{proof}

\subsection{Pancyclicity}

Recently the author, together with K\"uhn and
Osthus,~\cite{KKO_pan} showed that the minimum semi-degree
condition in Theorem~\ref{thm:ham_exact} gives not only a
Hamilton cycle but a cycle of every possible length. It is
natural to ask whether this can be extended to give all
orientations of all cycles of all possible lengths. A simple
probabilistic argument implies that Theorem~\ref{thm:arb_ham}
gives arbitrary orientations of any cycle of linear length (\ie
for all~$\alpha>0$, $\eta>0$ every sufficiently large oriented
graph~$G$ with~$\delta^0(G)\geq (3/8+\alpha)|G|$ contains every
orientation of any cycle of length at least~$\eta|G|$). It
remains an open question as to whether the error term can be
removed. The results on short cycles necessary to prove the
exact pancyclicity result in~\cite{KKO_pan} can (with the
addition of an error term in the minimum semi-degree condition)
also be extended to arbitrary orientations of cycles. In
particular, the following theorem can be obtained.

\begin{theorem}
Let $\alpha>0$. Then there exists~$n_0=n_0(\alpha)$ such that
if~$G$ is an oriented graph on~$n\geq n_0$ vertices with
minimum semi-degree~$\delta^0(G)\geq (3/8+\alpha)n$ then~$G$
contains a cycle of every possible orientation and of every
possible length.
\end{theorem}

A proof of Kelly, K\"uhn and Osthus of this result and
extensions of the stronger results on some short cycles can be
found in~\cite{KKO_pan}, along with a number of related open
problems.

\subsection{Overview of the Proof}
The proof of Theorem~\ref{thm:arb_ham} splits into
two parts, both relying on the expansion properties that our
minimum semi-degree condition implies. The cases are
distinguished by the similarity of the Hamilton cycle~$C$ we
are trying to embed to the standard orientation of a Hamilton
cycle. It turns out that the correct measure, at least for this
problem, of whether a cycle is close to a directed cycle is the
number of pairs of consecutive edges with different
orientations. Given an oriented graph~$C$ we call the subgraph
induced by three vertices~$x,y,z\in V(C)$ a \emph{neutral pair}
if~$xy,zy\in E(C)$. Given an arbitrarily oriented cycle~$C$
on~$n$ vertices let~$n(C)$ be the number of neutral pairs
in~$C$. Write~$C^*_n$ for the standard orientation of a cycle
on~$n$ vertices. When there is no ambiguity we will merely
write~$C^*$.

The essential idea is to split the cycle up into alternating
short and long paths and use the probabilistic method to find
an approximate embedding of the long paths into a Hamilton
cycle of the reduced graph created by applying a Regularity
lemma for digraphs. We connect these paths up greedily using
the short paths and then adjust the embedding to obtain
something which, after the Blow-up lemma has been applied,
gives us the desired orientation of a Hamilton cycle in our
graph.

The case distinction comes in the manner in which we alter our
embedding. In Section~\ref{sec:far} we give the argument for
cycles far from~$C^*$, where we use the neutral pairs for our
adjustments. In Section~\ref{sec:close} we assume that we have
few neutral pairs, and thus many long sections of~$C$
containing no changes in direction, and use these to adjust our
embedding.

The number of exceptional vertices that the Diregularity lemma
produces when applied directly is too great for the method used
here and hence some technical difficulties are introduced. So
we control the number of exceptional vertices by randomly
splitting our oriented graph~$G$. In still vague, but slightly
more precise terms, the Diregularity lemma will for
any~$\epsilon>0$ give us a partition with the property
of~$\epsilon$-regularity. It will also give us a set of
`exceptional vertices' which are in some sense badly behaved,
but tells us that these make up at most an~$\epsilon$
proportion of our vertices. Our method can only cope with~$\eta
n\ll\epsilon n$ such vertices. Hence we split the vertices of
our given graph~$G$ into two sets~$A$ and~$B$ of roughly equal
size (satisfying some `nice' properties). We apply the
Regularity lemma to~$G[B]$, giving us at most~$\epsilon |G|$
exceptional vertices~$V_0$. We then apply the Diregularity
lemma to~$G[A\cup V_0]$ only this time not with
parameter~$\epsilon$ but with~$\eta$. This gives us at
most~$\eta|G|$ exceptional vertices~$V_0'$. We then
consider~$G_B:=G[(B\sm V_0)\cup V_0']$, which
is~$\epsilon$-regular and has no exceptional vertices
and~$G_A:=G-G_B$, which is~$\eta$-regular and has~$0\ll
\eta|G_A|$ exceptional vertices. Hence, at the cost of some
technical work and having to stitch everything back together we
will be able to control the number of exceptional vertices.

The next section contains much of the notation we use in this
paper. In Section~\ref{sec:regularity_lemma} we introduce the
forms of the Diregularity lemma and Blow-up lemma that we need
later. In Section~\ref{sec:prep} we prepare the oriented
graph~$G$ and the cycle~$C$ for our approximate embedding and
in Section~\ref{sec:approx} prove the main tool needed to do
this. Following that in Section~\ref{sec:prep} we split into
our two cases and in Section~\ref{sec:far} ($C$ is far
from~$C^*$) and Section~\ref{sec:close} ($C$ is close to~$C^*$)
we prove Theorem~\ref{thm:arb_ham}.

\section{Notation}\label{sec:notation}

Given two vertices~$x$ and~$y$ of a digraph~$G$, we write~$xy$
for the edge directed from~$x$ to~$y$. The \emph{order}~$|G|$
of~$G$ is the number of its vertices. We write~$N^+_G(x)$ for
the outneighbourhood of a vertex~$x$ and $d^+(x):=|N^+_G(x)|$
for its outdegree. Similarly, we write~$N^-_G(x)$ for the
inneighbourhood of~$x$ and $d^-(x):=|N^-_G(x)|$ for its
indegree. Given $X\subseteq V(G)$ we denote $|N^+_G(x)\cap X|$
by $d^+_X(x)$, and define $d^-_X(x)$ similarly. We write
$N_G(x):=N^+_G(x)\cup N^-_G(x)$ for the neighbourhood of~$x$.
We use~$N^+(x)$ etc.~whenever this is unambiguous. We
write~$\Delta(G)$ for the maximum of~$|N(x)|$ over all
vertices~$x\in V(G)$. Given a set~$A$ of vertices of~$G$, we
write $N^+_G(A)$ for the set of all outneighbours of vertices
in~$A$. So~$N^+_G(A)$ is the union of $N^+_G(a)$ over all $a\in
A$. $N^-_G(A)$ is defined similarly. The directed subgraph
of~$G$ induced by~$A$ is denoted by~$G[A]$ and we write~$e(A)$
for the number of its edges. $G-A$ denotes the digraph obtained
from~$G$ by deleting~$A$ and all edges incident to~$A$.

Given two vertices $x,y$ of a digraph~$G$, an \emph{$x$-$y$
path} is a path with any orientation which joins~$x$ to~$y$. We
call a path with the standard  orientation a \emph{directed
path}. Given two subsets~$A$ and~$B$ of vertices of~$G$,
an~$A$-$B$ edge is an edge~$ab$ where $a\in A$ and $b\in B$. We
write $e(A,B)$ for the number of all these edges. A \emph{walk}
in~$G$ is a sequence $v_1v_2\dots v_{\ell}$ of (not necessarily
distinct) vertices, where~$v_iv_{i+1}$ or~$v_{i+1}v_i$ is an
edge for all $1\leq i<\ell$. The length of a walk~$W$
is~$\ell(W):=\ell-1$. The walk is \emph{closed} if
$v_1=v_\ell$. Given two vertices $x,y$ of~$G$, the
\emph{distance $dist(x,y)$ from~$x$ to~$y$} is the length of
the shortest directed $x$-$y$ path. The \emph{diameter} of $G$
is the maximum distance between any ordered pair of vertices.

We write~$[k]$ for the set~$\{1,2,\ldots,k\}$. We write $0<a_1
\ll a_2 \ll \ldots \ll a_k$ to mean that we can choose the
constants $a_1,a_2,\ldots,a_k$ from right to left. More
precisely, there are increasing functions
$f_1,f_2,\ldots,f_{k-1}$ such that, given $a_k$, whenever we
choose some $a_i \leq f_i(a_{i+1})$, all calculations needed
using these constants are valid.

\section{The Diregularity lemma and the Blow-up lemma}\label{sec:regularity_lemma}

In this section we collect all the information we need about
the Diregularity lemma and the Blow-up lemma. See~\cite{KSi}
for a survey on the Regularity lemma and~\cite{JKblowup} for a
survey on the Blow-up lemma. We start with some more notation.
The density of an undirected bipartite graph $G = (A,B)$ with
vertex classes~$A$ and~$B$ is defined to be
\[d_G(A,B) := \frac{e_G(A,B)}{|A||B|}.\]
We often write $d(A,B)$ if this is unambiguous. Given
$\epsilon>0$, we say that~$G$ is \emph{$\epsilon$-regular} if
for all subsets $X\subseteq A$ and $Y\subseteq B$ with
$|X|>\epsilon|A|$ and $|Y|>\epsilon|B|$ we have that
$|d(X,Y)-d(A,B)|<\epsilon$. Given $d\in [0,1]$ we say that~$G$
is $(\epsilon,d)$-\emph{super-regular} if it is $\eps$-regular
and furthermore $d_B(a)\ge (d-\eps) |B|$ for all $a\in A$ and
$d_A(b)\ge (d-\eps)|A|$ for all $b\in B$. (This is a slight
variation of the standard definition of
$(\epsilon,d)$-super-regularity where one requires $d_B(a)\ge d
|B|$ and $d_A(b)\ge d|A|$.)

The Diregularity lemma is a version of the Regularity lemma for
digraphs due to Alon and
Shapira~\cite{AlonShapiraTestingDigraphs}. Its proof is quite
similar to the undirected version. We will use the degree form
of the Diregularity lemma which can be easily derived (see
e.g.~\cite{young_05_extremal}) from the standard version, in
exactly the same manner as the undirected degree form.
\begin{lemma}[Degree form of the Diregularity lemma]\label{lemma:diregularity_lemma}
For every $\epsilon\in (0,1)$ and every integer~$M'$ there are
integers~$M$ and~$n_0$ such that if~$G$ is a digraph on $n\ge
n_0$ vertices and $d\in[0,1]$ is any real number, then there is
a partition of the vertices of~$G$ into $V_0,V_1,\ldots,V_k$
and a spanning subdigraph~$G'$ of~$G$  such that the following
holds:
\begin{itemize}
\item $M'\le k\leq M$,
\item $|V_0|\leq \epsilon n$,
\item $|V_1|=\cdots=|V_k|=:m$,
\item $d^+_{G'}(x)>d^+_G(x)-(d+\epsilon)n$ for all vertices $x\in G$,
\item $d^-_{G'}(x)>d^-_G(x)-(d+\epsilon)n$ for all vertices $x\in G$,
\item for every ordered pair~$V_iV_j$ with $1\le i,j\le k$
    and $i\neq j$ the bipartite graph $(V_i,V_j)_{G'}$
    whose vertex classes are~$V_i$ and~$V_j$ and whose edge
    set consists of all the $V_i$-$V_j$ edges in~$G'$ is
    $\epsilon$-regular and has density either~0 or at
    least~$d$,
\item for all~$1\leq i\leq k$ the digraph~$G'[V_i]$ is
    empty.
\end{itemize}
\end{lemma}

$V_1,\ldots,V_k$ are called \emph{clusters}, $V_0$ is called
the \emph{exceptional set} and the vertices in~$V_0$ are called
\emph{exceptional vertices}. Note that in~$G'$ all pairs of
clusters are $\epsilon$-regular in both directions (but
possibly with different densities). We call the spanning
digraph~$G'\subseteq G$ given by the Diregularity lemma the
\emph{pure digraph}. Given clusters $V_1,\ldots,V_k$ and the
pure digraph~$G'$, the \emph{reduced digraph~$R'$} is the
digraph whose vertices are $V_1,\ldots,V_k$ and in which
$V_iV_j$ is an edge if and only if~$G'$ contains a $V_i$-$V_j$
edge. Note that the latter holds if and only if
$(V_i,V_j)_{G'}$ is $\epsilon$-regular and has density at
least~$d$. It turns out that~$R'$ inherits many properties
of~$G$, a fact that is crucial in our proof. However,~$R'$ is
not necessarily oriented even if the original digraph~$G$ is.
The following straightforward lemma, taken from a paper of
Kelly, K\"uhn and Osthus~\cite{kelly_kuhn_osthus_hc_orient},
shows that by discarding edges with appropriate probabilities
one can go over to a reduced oriented graph $R\subseteq R'$
which still inherits many of the properties of~$G$.
\begin{lemma}\label{lemma:reduced_oriented}
For every $\epsilon\in (0,1)$ there exist
integers~$M'=M'(\epsilon)$ and $n_0=n_0(\epsilon)$ such that
the following holds. Let $d \in [ 0,1 ]$, let~$G$ be an
oriented graph of order at least~$n_0$ and let~$R'$ be the
reduced digraph and~$G'$ the pure digraph obtained by applying
the Diregularity lemma to~$G$ with parameters $\epsilon$, $d$
and~$M'$. Then~$R'$ has a spanning oriented subgraph~$R$ with
\begin{itemize}
\item[{\rm (a)}] $\delta^+(R)\ge
    (\delta^+(G)/|G|-(3\epsilon+d))|R|$,
\item[{\rm (b)}]
    $\delta^-(R)\ge(\delta^-(G)/|G|-(3\epsilon+d))|R|$,
\item[{\rm (c)}]
    $\delta^0(R)\ge(\delta^0(G)/|G|-(6\epsilon+4d))|R|$.
\end{itemize}
\end{lemma}

The oriented graph~$R$ given by
Lemma~\ref{lemma:reduced_oriented} is called the \emph{reduced
oriented graph}. The spanning oriented subgraph~$G^*$ of the
pure digraph~$G'$ obtained by deleting all the $V_i$-$V_j$
edges whenever $V_iV_j\in E(R')\setminus E(R)$ is called the
\emph{pure oriented graph}. Given an oriented subgraph
$S\subseteq R$, the \emph{oriented subgraph of~$G^*$
corresponding to~$S$} is the oriented subgraph obtained
from~$G^*$ by deleting all those vertices that lie in clusters
not belonging to~$S$ as well as deleting all the $V_i$-$V_j$
edges for all pairs $V_i,V_j$ with $V_iV_j\notin E(S)$.

At various stages in our proof we will need some pairs of
clusters to be not just regular but super-regular. The
following well-known result tells us that we can indeed do this
whilst maintaining the regularity of all other pairs.

\begin{lemma}\label{lemma:super}
Let~$\epsilon\ll d,1/\Delta$ and let~$R$ be a reduced oriented
graph of~$G$ as given by
Lemmas~\textup{\ref{lemma:diregularity_lemma}}
and~\textup{\ref{lemma:reduced_oriented}}. Let~$S$ be an
oriented subgraph of~$R$ of maximum degree~$\Delta$. Then we
can move exactly~$2\Delta\epsilon |V_i|$ vertices from each
cluster into~$V_0$ such that each pair~$(V_i,V_j)$
corresponding to an edge of~$S$ becomes
$(2\epsilon,d/2)$-super-regular and every pair corresponding to
an edge of~$R\sm S$ becomes~$2\epsilon$-regular with density at
least~$d-\epsilon$.
\end{lemma}

In our proof of Theorem~\ref{thm:arb_ham} we will also need a
consequence of the Blow-up lemma of Koml\'os, S\'ark\"ozy and
Szemer\'edi~\cite{Komlos_Szemeredi_Blowup_Lemma}. Roughly
speaking, it says that an $r$-partite graph formed by~$r$
clusters such that all the pairs of these clusters are
$(\eps,d)$-super-regular behaves like a complete $r$-partite
graph with respect to containing graphs of bounded maximum
degree as subgraphs.

\begin{lemma}[Blow-up Lemma, Koml\'os, S\'ark\"ozy and
Szemer\'edi~\cite{Komlos_Szemeredi_Blowup_Lemma}]\label{standardblowup}
Given a graph $F$ on $[k]$ and positive integers~$d$ and
$\Delta$ there exists a positive real
$\epsilon_0=\epsilon_0(d,\Delta,k)$ such that the following
holds for all positive numbers $\ell_1,\dots,\ell_k$ and all
$0<\epsilon\le \epsilon_0$. Let $F'$ be the graph obtained from
$F$ by replacing each vertex $i\in F$ with a set $V_i$ of
$\ell_i$ new vertices and joining all vertices in $V_i$ to all
vertices in $V_j$ whenever $ij$ is an edge of $F$. Let $G'$ be
a spanning subgraph of $F'$ such that for every edge $ij\in F$
the graph $(V_i,V_j)_{G'}$ is $(\epsilon,d)$-super-regular.
Then $G'$ contains a copy of every subgraph $H$ of $F'$ with
maximum degree $\Delta(H)\le \Delta$. Moreover, this copy of
$H$ in $G'$ maps the vertices of $H$ to the same sets $V_i$ as
the copy of $H$ in~$F'$, i.e.~if $h \in V(H)$ is mapped to
$V_i$ by the copy of~$H$ in~$F'$, then it is also mapped to
$V_i$ by the copy of $H$ in~$G'$.
\end{lemma}

The tool we shall actually use is the following consequence of
the Blow-up lemma. The proof of it uses similar ideas to those
in recent work of Christofides, Keevash, K\"uhn and
Osthus~\cite{CKKO}.

\begin{lemma}\label{lemma:blowupused}
Suppose that all the following hold:
\begin{itemize}
\item $0<1/m\ll \epsilon \ll d \ll 1$.
\item $U_1,\ldots,U_k$ are pairwise disjoint sets of
    size~$m$, for some~$k\geq 6$, and~$G$ is a digraph
    on~$U_1\cup\ldots \cup U_k$ such that
    each~$(U_i,U_{i+1})_G$ is~$(\epsilon,d)$-super-regular
    (where by convention we consider~$U_{k+1}$ to
    be~$U_1$);
\item $A_1,\ldots,A_k$ are pairwise disjoint sets of
    vertices with~$(1-\epsilon)m \leq |A_i|=:m_i\leq m$
    and~$H$ is a digraph on~$A_1\cup\ldots \cup A_k$ which
    is a vertex-disjoint union of paths of length at least
    \textup{3}, where every edge going out of~$A_i$ ends
    in~$A_{i+1}$ for all~$i$;
\item $S_1\subseteq U_1,\ldots ,S_k\subseteq U_k$ are sets
    of size~$|S_i|=m_i$;
\item For each path~$P$ of~$H$ we are given
    vertices~$x_P,y_P\in V(G)$ such that if the initial
    vertex~$a_P$ of~$P$ belongs to~$A_i$ then~$x_P\in S_i$
    and if the final vertex~$b_P$ of~$P$ belongs to~$A_j$
    then~$y_P\in S_j$, and the vertices~$x_P,y_P$ are
    distinct as~$P$ ranges over the paths of~$H$.
\end{itemize}
Then there is an embedding of~$H$ into~$G_S:=G[\bigcup S_i]$ in
which every path~$P$ of~$H$ is mapped to a path that starts
at~$x_P$ and ends at~$y_P$.
\end{lemma}

The following immediate consequence of the Blow-up lemma is
needed in the proof of Lemma~\ref{lemma:blowupused}.

\begin{lemma}\label{lemma:linking}
For every~$0<d<1$ and~$p\geq 4$ there exists~$\epsilon_0>0$
such that the following holds for~$0<\epsilon<\epsilon_0$.
Let~$U_1,\ldots,U_p$ be pairwise disjoint sets of size~$m$, for
some~$m$, and suppose~$G$ is a graph on~$U_1\cup\ldots\cup U_p$
such that each pair~$(U_i,U_{i+1})$, $1\leq i\leq p-1$ is
$(\epsilon,d)$-super-regular. Let~$f:U_1\rightarrow U_p$ be any
bijective map. Then there are~$m$ vertex-disjoint paths
from~$U_1$ to~$U_p$ so that for every~$x\in U_1$ the path
starting at~$x$ ends at~$f(x)\in U_p$.
\end{lemma}

We also need the following random partitioning property of
super-regular pairs which says that with high probability (\ie
with probability tending to 1 as $m\rightarrow \infty$) all new
pairs created by a random partition of a super-regular pair are
themselves super-regular.\COMMENT{Does this/will this appear in
another paper or count as well-known? Or else should I add a
proof. It's quick and easy but I'd prefer to present it as
fact.}

\begin{lemma}\label{lemma:splitsuperreg}
Suppose that the following hold.
\begin{itemize}
\item $0<\epsilon < \theta< d < 1/2$, $k\geq 2$ and for
    $1\leq i\leq k$ we have~$a_i,b_i>\theta$ with
    $\sum_{i=1}^k a_i= \sum_{i=1}^k b_i=1$.
\item $G=(A,B)$ is an $(\epsilon,d)$-super-regular pair
    with $|A|=|B|=m$ sufficiently large.
\item $A=A_1\cup \ldots \cup A_k$ and $B=B_1\cup \ldots
    \cup B_k$ are partitions chosen uniformly at radnom
    with~$|A_i|=a_i m$ and~$|B_i|=b_i m$ for~$1\leq i\leq
    r$.
\end{itemize}
Then with high probability $(A_i,B_j)$ is
$(\theta^{-1}\epsilon,d/2)$-super-regular for every~$1\leq
i,j\leq k$. \qed
\end{lemma}

With these tools we can now prove Lemma~\ref{lemma:blowupused}.

\begin{proof}[of Lemma~\ref{lemma:blowupused}]
Enumerate the paths of~$H$ as~$P_1,\ldots,P_p$ and split them
up arbitrarily into paths of length 3, 4 or 5 such that~$P_i$
becomes~$P_{i,1},\ldots,P_{i,q_i}$. Let~$a_{i,j}$ and~$b_{i,j}$
be the initial vertex and the final vertex of~$P_{i,j}$
respectively. Then~$a_{i,j}=b_{i,j-1}$ for~$2\leq j\leq q_i$.
Let~$E_s$ consist of all~$a_{i,j}$ belonging to the
cluster~$A_s$ and similarly let~$F_s$ consist of all~$b_{i,j}$
belonging to the cluster~$A_s$. For each~$a_{i,j}\in E_s$ pick
a distinct vertex~$x_{i,j}\in S_s$ and for each~$b_{i,j}\in
F_s$ pick a distinct vertex~$y_{i,j}\in S_s$ such that
if~$a_{i,j}=b_{i,j-1}$ then~$x_{i,j}=y_{i,j-1}$,
$x_{i,1}=x_{P_{i,j}}$ and~$y_{i,m_i}=y_{P_{i,j}}$. It is
sufficient to show that there is an embedding of~$H$ in which
each path~$P_{i,j}$ is mapped to a path in~$G_S$ starting
at~$x_{i,j}$ and ending at~$y_{i,j}$.

For a path~$P_{i,j}$ encode whether each edge in~$P_{i,j}$ goes
forwards or backwards. If~$P_{i,j}$ has length~3 then, writing
\texttt{f} for an edge going from some~$A_{\ell}$
to~$A_{\ell+1}$ and \texttt{b} for an edge going
from~$A_{\ell}$ to~$A_{\ell-1}$, $t$ encodes one of the
following~$2^3=8$ possibilities:
\[\text{\texttt{fff ffb fbf fbb bff bfb bbf bbb.}}\]
Similarly there are~$2^4$ possibilities for paths of length 4
and~$2^5$ for those of length 5. We divide the paths~$P_{i,j}$
into~$56k$ subcollections based on the orientations of their
edges. It transpires that there are notational advantages in
doing this by encoding the destination of each vertex relative
to the first. More precisely, we divide the paths into
subcollections $\P_{i,t}$ with~$1\leq i\leq k$, $3\leq \ell
\leq 5$ and
\[t:\{0,1,\ldots,\ell\}\rightarrow \{-\ell,-\ell+1,\ldots,\ell\}\]
encoding one of the~$2^3+2^4+2^5=56$ possibilities discussed
above and the length~$\ell=\ell(t)$ of the paths. Note that we
always have~$t(0)=0$. For example, a path oriented~\texttt{ffb}
would have~$t:(0,1,2,3)\mapsto(0,1,2,1)$. $\P_{i,t}$ contains
all paths~$P_{i,j}$ of length~$\ell$ starting in~$A_i$ with
each vertex in~$P_{i,j}$ going to the cluster relative to~$A_i$
given by~$t$.

Observe that as~$|U_i\sm S_i| \leq \epsilon m$, every
pair~$(S_i,S_{i+1})$ is $(2\epsilon,d/2)$-super-regular. We
first use a greedy algorithm to sequentially embed those
collections~$\P_{i,t}$ containing at most~$d^2m$ paths. That
is, we pick any~$|\P_{i,t}|$ vertices in~$S_i$ to be the start
of these paths, and then construct these paths by selecting any
(distinct) neighbours of these vertices in the~$S_j$
appropriate for each vertex in each path. Each set~$S_i$ is met
by at most~$11\times 56$ of the collections so at any stage in
this process we have used at most~$6\times 11\times 56d^2m$
vertices\COMMENT{Actually a fair bit smaller than this,
$304=4*2^3+5*2^4+6*2^5$ certainly suffices and can be improved
upon, but this neither matters nor makes matters any clearer
than the given rough bound.} from any cluster~$U_i$. As we
have~$d\ll 1$ the restriction of any pair~$(S_i,S_{i+1})$ to
the remaining vertices is still $(4\epsilon,d/4)$-super-regular
and so we can indeed do this.

Having embedded all the~$\P_{i,t}$ containing few paths, we
randomly split the remaining vertices so that for each
large~$\P_{i,t}$ we have sets~$S_{i,t}^0\subseteq
S_{i+t(0)=i}$, $S_{i,t}^1\subseteq S_{i+t(1)}$, \ldots,
$S_{i,t}^\ell\subseteq S_{i+t(\ell)}$ each of
size~$|\P_{i,t}|>d^2m$.\COMMENT{Note that, for example,~$t(0)$
could equal~$t(2)$.} By Lemma~\ref{lemma:splitsuperreg} for
each large collection~$\P_{i,t}$ and for all~$0\leq r\leq
\ell-1$ the pair~$(S_{i,t}^r,S_{i,t}^{r+1})$ if~$t(r+1)>t(r)$
or the pair~$(S_{i,t}^{r+1},S_{i,t}^{r})$ if~$t(r+1)<t(r)$ is
$(4d^{-2}\epsilon,d/8)$-super-regular with high probability.
Thus for sufficiently large~$m$ we can choose a partition with
this property and apply Lemma~\ref{lemma:linking} to embed each
large~$\P_{i,t}$ within its allocated sets.
\end{proof}

\section{Skewed Traverses and Shifted Walks}\label{sec:ss}

In this section we introduce some tools needed to tweak a
random embedding of an arbitrarily oriented Hamilton cycle into
a directed Hamilton cycle of the reduced oriented graph to make
it correspond (in some sense) to the desired orientation of a
Hamilton cycle in our original graph.

The following crucial result says that our minimum semi-degree
condition implies outexpansion.

\begin{lemma}[Kelly, K\"uhn, Osthus \cite{kelly_kuhn_osthus_hc_orient}]\label{lemma:expansion}
Let~$R$ be an oriented graph with $\delta^0(R)\geq
(3/8+\alpha)|R|$ for some $\alpha>0$. If $X\subset V(R)$ with
$0<|X|\leq (1-\alpha)|R|$ then $|N^+(X)| \ge |X|+\alpha |R|/2.$
\end{lemma}

Suppose that~$F$ is a Hamilton cycle (with the standard
orientation) of the reduced oriented graph~$R$ and relabel the
vertices of~$R$ such that~$F=V_1V_2\ldots V_M$, where we
let~$M:=|R|$. Create a new digraph $R^*$ from~$R$ by adding all
the exceptional vertices $v\in V_0$ to~$V(R)$ and adding an
edge~$vV_i$ (where~$V_i$ is a cluster containing~$m$ vertices)
whenever~$|N^+_{V_i}(v)|\geq cm$ for some given constant~$c>0$.
(Recall that~$m$ denotes the size of the clusters.) The edges
in~$R^*$ of the form~$V_iv$ are defined in a similar way.
Let~$G^c$ be the digraph obtained from the pure oriented
graph~$G^*$ by making all the non-empty bipartite subgraphs
between the clusters complete (and orienting all the edges
between these clusters in the direction induced by~$R$) and
adding the vertices in~$V_0$ as well as all the edges of~$G$
between~$V_0$ and $V(G-V_0)$.

Let~$W$ be an assignment of the vertices of an arbitrarily
oriented cycle~$C$ on~$n$ vertices to the vertices of~$R^*$
which respects edges (i.e.\@ is a digraph homomorphism from~$C$
to~$R^*$). We denote by~$a(i)$ the number of vertices of~$C$
assigned to the cluster~$V_i$. Observe that we can think of~$W$
either as a (possibly degenerate) embedding into~$G^c$ or as a
closed walk in~$R^*$. It will be useful to the reader to keep
this duality in mind when reading the rest of the proof We say
that an assignment~$W$ of~$C$ to~$R^*$ is
\emph{$\gamma$-balanced} if $\max_i|a(i)-m|\leq\gamma n$ and
\emph{balanced} if~$a(i)=m$ for all~$1\leq i\leq M$.
Furthermore, we say that an assignment
\emph{$(\gamma,\mu)$-corresponds to~$C$} if the following
conditions hold.
\begin{itemize}
\item $W$ is $\gamma$-balanced.
\item Each exceptional vertex~$v\in V_0$ has exactly one
    vertex of~$C$ assigned to it.
\item In every~$V_i\in V(R)$ at least~$m-\mu n$ of the
    vertices of~$C$ assigned to~$V_i$ have both of their
    neighbours assigned to~$V_{i-1}\cup V_{i+1}$.
\end{itemize}
We say that the assignment \emph{$\mu$-corresponds to~$C$} if
it $(0,\mu)$-corresponds to~$C$.\COMMENT{Use~$n$ not~$m$ here
as it makes things look nicer later in the proof.}

Once we have found such an assignment we can, with some work,
use Lemma~\ref{lemma:blowupused} to show that it corresponds to
a copy of~$C$ in~$G$. Our immediate aim then is to find such a
closed walk corresponding to~$C$.
%

\begin{figure}
  \centering\footnotesize
  \includegraphics[scale=0.9]{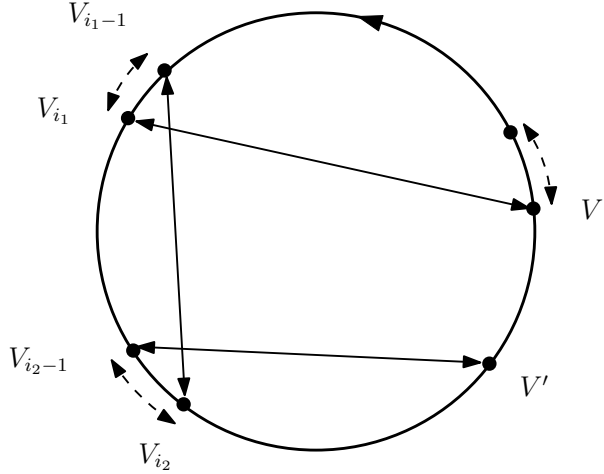}
  \caption{A skewed $V$-$V'$ traverse}\label{fig:skewed}
\end{figure}

Given clusters~$V$ and~$V'$, a \emph{skewed $V$-$V'$
traverse}~$T(V,V')$ is a collection of edges of the form
\[ T(V,V'):=VV_{i_1},V_{i_1-1}V_{i_2},V_{i_2-1}V_{i_3},\ldots ,V_{i_t-1}V'.\]
The \emph{length} of a skewed traverse in the number of its
edges minus one; so the length of the above skewed traverse
is~$t$.\COMMENT{This definition is such that the length of a
skewed traverse is the same as the number of cycles traversed
by a shifted walk derived from it.} Suppose that we have
a~$\gamma$-balanced assignment~$W$ of~$C$ to~$R^*$ and that
each vertex of~$R$ has many neutral pairs of~$C$ assigned to
it. We would like to make this a balanced embedding by
modifying~$W$. Let~$V_i,V_j$ be clusters with~$a(i)>m$
and~$a(j)<m$. If~$V_{i-1}V_j\in E(R)$ then we could replace one
neutral pair assigned to~$V_{i-1}V_iV_{i-1}$ in the embedding
with~$V_{i-1}V_jV_{i-1}$. This would reduce~$a(i)$ by one and
increase~$a(j)$ by one. Repeating this process would give the
desired balanced embedding. We can not guarantee though
that~$V_{i-1}V_j\in E(R)$ so we are forced to use skewed
traverses to achieve the same effect, which we are able to show
always exist under certain conditions. Let
\[ V_{i-1}V_{i_1},V_{i_1-1}V_{i_2},V_{i_2-1}V_{i_3},\ldots ,V_{i_t-1}V_j.\]
be a skewed $V_{i-1}$-$V_j$ traverse. Then replacing neutral
pairs starting at~$V_{i-1},V_{i_1-1},\ldots,V_{i_t-1}$ with the
edges in the skewed $V_{i-1}$-$V_j$ traverse we reduce~$a(i)$
by one, increase~$a(j)$ by one and crucially do not
alter~$a(k)$ for any~$V_k\in V(R)\sm \{V_i,V_j\}$. See
Figure~\ref{fig:skewed} for an illustration of this, where the
dashed edges represent the neutral pairs which will be replaced
by the solid edges representing the edges of the skewed
traverse. We always assume that a skewed traverse has minimal
length and thus that each vertex~$V_i\in V(R)$ appears at most
once as the first vertex of an edge in a skewed traverse.

Given vertices~$V,V'\in V(R)$ and a Hamilton cycle~$F$ of~$R$,
a \emph{shifted $V$-$V'$ walk}~$S(V,V')$ is a walk of the form
\[
S(V,V'):=V\,V_{i_1}FV_{i_1-1}\,V_{i_2}FV_{i_2-1}\ldots V_{i_t}FV_{i_t-1}\,V',
\]
where we write~$V_iFV_j$ for the path
\[
V_iFV_j:=V_iV_{i+1}V_{i+2}\ldots V_j,
\]
counting indices modulo~$|F|=k$. (The case~$t=0$, and thus a
walk $VV'$, is allowed.) We say that~$W$ traverses~$F$~$t$
times and always assume that a shifted walk~$S(V,V')$
traverses~$F$ as few times as possible. Its length is the
length of the corresponding walk in~$R$. Note that if we can
find a skewed~$V$-$V'$ traverse then we can find a shifted
$V$-$V'$~walk.

The most important property of shifted walks is that the
walk~$W-\{V,V'\}$ visits every vertex in~$R$ an equal number of
times. Observe also that by our minimality assumption each
vertex~$V_i$ is visited at most one time from a vertex other
than~$V_{i-1}$. I.e.\@ of the~$t$ times that~$V_i$ is visited
at most one does not come from winding around~$F$. This fact
will be useful later when we try and bound the number of edges
of an embedding not lying on the edges of~$F$.

As with skewed traverses, we can use shifted walks to go from
an approximate assignment~$W$ of a cycle~$C$ to a balanced
assignment. Let~$V_i,V_j$ be clusters with~$a(i)>m$
and~$a(j)<m$. If~$V_{i-1}V_j,V_jV_{i+1}\in E(R)$ then we could
replace one section of~$W$ isomorphic to~$F$
by~$V_{i-1}V_jV_{i+1}FV_{i-1}$, that is,
replace~$V_{i-1}V_iV_{i+1}$ by~$V_{i-1}V_jV_{i+1}$. This new
section has the same length as before and so would not alter
the rest of~$W$. Clearly we can not ensure that such edges
always exist. Instead we use shifted walks and replace a
section of the embedding that looks like~$FF\ldots F$ with
\[
S(V_{i-1},V_j)S(V_{j},V_{i+1})FV_{i-1}F\ldots FV_{i-1};
\]
where the~$F\ldots F$ in the new embedding contains the
appropriate number of~$F$ to ensure that it is of exactly the
same length as the section of the assignment it replaced. This
is a shifted walk from~$V_{i-1}$ to~$V_j$, then a shifted walk
from~$V_{j}$ to~$V_{i+1}$ and then wind around~$F$. By our
definition of shifted walks each cluster will have the same
number of vertices assigned to it (except~$V_{i-1}$, $V_i$
and~$V_j$) and the total number of vertices assigned will not
be altered. Clearly this method needs the cycle we're trying to
embed to contain many long sections with no changes of
orientation (and oriented in the same direction as~$F$). In the
case where the cycle we are trying to embed is close to~$C^*$,
the standard orientation of a cycle, we are indeed able to
ensure this.

\begin{corollary}\label{cor:join}
Let~$R$ be an oriented graph on $k$ vertices with
$\delta^0(R)\geq(3/8+\alpha)k$ for some $\alpha>0$ and
let~$F=V_1V_2\ldots V_k$ be a directed Hamilton cycle of~$R$.
Define~$r:=\lceil 2/\alpha\rceil$. Then for any
distinct~$V,V'\in V(R)$ there exists the following.
\begin{itemize}
    \item[(i)] A skewed $V$-$V'$ traverse of length at
        most~$r$.
    \item[(ii)] A shifted $V$-$V'$ walk traversing at
        most~$r$ cycles.
\end{itemize}
\end{corollary}

\begin{proof}
Let~$A_i$ be the set of vertices which can be reached from~$V$
by a skewed traverse of length at most~$i$ and
let~$A_i^-:=\{V_i\in V(R):V_{i+1}\in A_i\}$.
If~$|A_i|\geq(1-\alpha)k$ then~$N^-(V')\cap A_i^- \neq
\emptyset$ and we have a skewed~$V$-$V'$ traverse of
length~$i+1$. If~$|A_i|\leq(1-\alpha)k$ then we can apply
Lemma~\ref{lemma:expansion} (here we also need that
$N^+(V)\neq\emptyset$) to get that~$|A_{i+1}|\geq |A_i|+\alpha
k/2$. Since~$|A_{r-2}|>(1-\alpha)k>k-|N^-(V')|$ we again
have~$N^-(V')\cap A_{r-2}^- \neq \emptyset$ and hence the
desired skewed traverse.

This skewed traverse also gives the desired shifted walk,
merely `wind around'~$F$ after each edge.
\end{proof}

When linking together sections of our cycle we will sometimes
need to find a path between two vertices which is not just
short but is isomorphic to a path with given length and
orientation. To do this we use the following lemma of
H\"aggkvist and Thomason.

\begin{lemma}[H\"aggkvist and Thomason
\cite{HaggkvistThomasonHamilton}]\label{lemma:exact_length}
Let~$R$ be an oriented graph on~$k$ vertices
with~$\delta^0(R)\geq(3/8+\alpha)k$ for some~$\alpha>0$. Let
$4\lceil\log_2(1/\alpha)\rceil\leq k\leq \alpha k/4$ and
let~$P$ be an arbitrarily oriented path of length~$k$. Then,
if~$k$ is large enough and $V,V'\in V(R)$ are distinct
vertices, there exists a path from~$V$ to~$V'$ isomorphic
to~$P$.
\end{lemma}

\section{An approximate embedding lemma}\label{sec:approx}

Our main tool in our proof of Theorem~\ref{thm:arb_ham} is the
following probabilistic result which says that we can assign a
series of paths~$P_i$ to the vertices of a small graph~$R$ such
that each vertex of~$R$ is assigned approximately the same
number of vertices. Furthermore, we show that if we have a
collection of subpaths of the~$P_i$ we can assure that every
vertex of~$R$ is assigned a reasonable number of the starting
points of these. When we talk about `greedily embedding an
oriented path~$P_i$ around a cycle~$F$ given a starting
point~$V\in V(F)$' we mean the following. Assign the first
vertex of~$P_i$ to~$V$. Given an embedding of some initial
segment of~$P_i$ which ends at~$V'\in V(F)$ assign the next
vertex of~$P_i$ to either the successor or the predecessor
of~$V'$ in~$F$ according to the orientation of the edge
in~$P_i$.

\begin{lemma}\label{lemma:embed}
Let~$R$ be an oriented graph on $k$ vertices and let~$F$ be a
Hamilton cycle in~$R$. Let~$\P=\{P_1,\ldots,P_s\}$ be a
collection of arbitrarily oriented paths on~$t$ vertices and
$\Q$ be a collection of pairwise disjoint oriented subpaths of
the~$P_i$. Then for any~$\gamma>0$ and sufficiently large $s$
there exists a map $\phi:[s]\rightarrow V(R)$ such that if the
paths are greedily embedded around~$F$ with the embedding of
each~$P(i)$ starting at~$\phi(i)$ then the following holds.
Define~$a(i)$ to be the number of vertices
in~$\bigcup_{j=1}^{s}P_j$ assigned to~$V_i$ by this embedding
and define~$n(i,\Q)$ to be the number of oriented subpaths
in~$\Q$ starting at~$V_i$. Then for all~$V_i\in V(R)$
\begin{gather}
\left|a(i)-\frac{st}{k}\right| \leq \gamma st, \\
\left|n(i,\Q)-\frac{|\Q|}{k}\right| \leq \gamma
st.\label{eq:splitQ}
\end{gather}
\end{lemma}

To prove it we need the following well-known probabilistic
bound (see~\cite{graph_colouring} for example).


\begin{theorem}\label{thm:azuma}
Let~$X$ be a random variable determined by~$s$ independent
trials $X_1,\ldots,X_s$ such that changing the outcome of any
one trial can affect~$X$ by at most~$c$. Then for
any~$\lambda>0$,
\[
\Pr(|X-\E(X)| > \lambda) \leq 2\exp\left({-\frac{\lambda^2}{2c^2s}}\right).
\]
\end{theorem}

\begin{proof}[of Lemma~\ref{lemma:embed}]
We construct~$\phi$ by picking each~$\phi(i)$ independently and
uniformly at random. Observe that the assignment of any one
path~$P_i$ can change the number of vertices assigned to any
vertex of~$R$ by at most~$t$. Clearly~$\E(a(i))=st/k$. By
Theorem~\ref{thm:azuma} we have
\begin{equation*}
    \Pr(|a(i)-st/k| > \gamma st) \leq 2\exp(-\frac{\gamma^2s^2t^2}{2t^2s})
    =  2\exp(-\frac{\gamma^2s}{2}) < 1/(2k)
\end{equation*}
for~$s\gg k$.

A similar argument gives that the probability that~$n(i,\Q)$
differs too much from the expected value is at most~$1/(2k)$.
Thus the probability that there exists~$V_i$ which does not
have almost the expected number of vertices or almost the
expected number of starting points of paths in~$\Q$ assigned to
it by~$\phi$ is less than~$1$. So with positive probability a
map constructed in this manner satisfies the conclusion of the
lemma, and hence such a map exists.
\end{proof}

\section{Preparations for the Proof of
Theorem~\ref{thm:arb_ham}}\label{sec:prep}

\subsection{The Two Cases}\label{sec:cases}

We split into two cases depending on the number of neutral
pairs. Let~$G$ be an oriented graph on~$n$ vertices
with~$\delta^0(G)\geq (3/8+\alpha)n$ for some
constant~$0<\alpha\ll 1$.\COMMENT{Should it be noted somewhere
that we can assume that~$\alpha$ is small?} Let~$C$ be an
orientation of a cycle on~$n$ vertices with~$n(C)=:\lambda n$
neutral pairs. Define the following hierarchy of constants.
\[
0 < \eps_1 \ll \eps_2 \ll \eps_3 \ll \eps_4 \ll \eps_5 \ll \eps_6 \ll \alpha < 1.
\]
Let~$\Q$ be a maximal collection of neutral pairs all at a
distance of at least~3 from each other.

If~$\lambda\ll \epsilon_4$ then let~$\eps:=\eps_6$,
$\eps_A:=\eps_5$ and~$\eps^*:=\eps_4$. The proof of this case
is given in Section~\ref{sec:close}.

Otherwise we have~$\lambda \gg \epsilon_3$ and we
set~$\eps:=\eps_3$, $\eps_A:=\eps_2$ and~$\eps^*:=\eps_1$. The
proof of this case is in Section~\ref{sec:far}.

The following two sections, where we partition~$G$ and~$C$ in
preparation for our embedding, are common to both cases.

\subsection{Preparing~$G$ for the Proof of
Theorem~\ref{thm:arb_ham}}\label{sec:prepG}

Define a positive constant~$d$ and integers~$M'_A,M'_B$ (all
functions of~$\alpha$) such that
\[
0 < \eps^* \ll 1/M'_A \ll \epsilon_A \ll 1/M'_B \ll \epsilon \ll d \ll \alpha \ll 1.
\]
Chernoff type bounds applied to a random partition of~$V(G)$
show the existence of a subset~$A\subset V(G)$ with
$(1/2-\epsilon)n \leq |A| \leq (1/2-\epsilon)n$ such that every
vertex~$x\in V(G)$ satisfies
\begin{equation}\label{eq:d+split}
\frac{d^+(x)}{n}-\frac{\alpha}{10} \leq \frac{|N^+_A(x)|}{|A|} \leq \frac{d^+(x)}{n}+\frac{\alpha}{10}
\end{equation}
and similarly for~$d^-(x)$. Apply the Diregularity lemma
(Lemma~\ref{lemma:diregularity_lemma}) with
parameters~$\epsilon^2$, $d+8\eps^2$ and~$M'_B$ to~$G- A$ to
obtain a partition of the vertex set of~$G-A$ into $M_B:=k\ge
M'_B$ clusters $V_1,\dots,V_{k}$ and an exceptional set~$V_0$.
Set $B:=V_1\cup\ldots\cup V_{k}$ and
$m_B:=|V_1|=\dots=|V_{k}|$. Let~$G_B':=G[B]$, let~$R_B$ denote
the reduced oriented graph obtained by an application of
Lemma~\ref{lemma:reduced_oriented} and let~$G^*_B$ be the pure
oriented graph. By our choice of~$A$ we have
$\delta^+(G-A)/|G-A|\ge \delta^+(G)/n-\alpha/9$ and a similar
bound for~$\delta^-$. Hence we can apply
Lemma~\ref{lemma:reduced_oriented} to obtain
\begin{equation}\label{eq:delta0_RB}
\delta^0(R_B)\geq \left(\frac{\delta^0(G)}{n}-\frac{\alpha}{4}\right)|R_B|
    \geq \left(\frac{3}{8}+\frac{3\alpha}{4}\right)|R_B|.
\end{equation}
So Theorem~\ref{thm:ham_exact} gives us a Hamilton cycle~$F_B$
of~$R_B$. Relabel the clusters of~$R_B$ so that~$V_iV_{i+1}\in
E(F_B)$ for all~$i$ where we let~$V_{k+1}:=V_1$. We now apply
Lemma~\ref{lemma:super} with~$F_B$ playing the role of~$S$,
$\eps^2$ playing the role of~$\eps$ and~$d+8\eps^2$ playing the
role of~$d$. This shows that by adding at most~$4\eps^2 n$
further vertices to the exceptional set~$V_0$ we may assume
that each edge of~$R_B$ corresponds to an $\eps$-regular pair
of density at least~$d$ (in the underlying graph of~$G^*_B$)
and that each edge in~$F_B$ corresponds to an $(\eps,
d)$-super-regular pair. Note that the new exceptional set now
satisfies $|V_0|\le \eps n$.

Now apply the Diregularity Lemma with parameters $\eps_A^2/4$,
$d+2\eps_A^2$ and~$M'_A$ to~$G[A\cup V_0]$ to obtain a
partition of the vertex set of~$G[A\cup V_0]$ into
$M_A:=\ell\ge M'_A$ clusters $V'_1,\dots,V'_\ell$ and an
exceptional set~$V'_0$. Let $A':=V'_1\cup\dots\cup V'_\ell$,
let~$R_{A}$ denote the reduced oriented graph obtained from
Lemma~\ref{lemma:reduced_oriented} and let~$G^*_{A}$ be the
pure oriented graph. As before
Lemma~\ref{lemma:reduced_oriented} implies that
$\delta^0(R_A)\geq(3/8+3\alpha/4)|R_A|$ and so, as before, we
can apply Theorem~\ref{thm:ham_exact} to find a Hamilton
cycle~$F_A$ of~$R_A$. Then as before, Lemma~\ref{lemma:super}
implies that by adding at most~$\eps_A^2 |A\cup V_0|$ further
vertices to the exceptional set~$V'_0$ we may assume that each
edge of~$R_A$ corresponds to an $\eps_A$-regular pair of
density at least~$d$ and that each edge in~$F_A$ corresponds to
an $(\eps_A, d)$-super-regular pair. Finally
define~$G_B:=G[B\cup V_0']$ and~$n_B:=|G_B|$ and observe that
we now have
\begin{equation}\label{eq:V'0}
|V'_0|\leq \epsilon_A |A\cup V_0|/2 < \epsilon_A n_B.
\end{equation}
In both cases of our proof we now have
\[
0 < \eps^* \ll 1/M_A \ll \epsilon_A \ll 1/M_B \ll \epsilon \ll d \ll \alpha \ll 1.
\]

\subsection{Preparing~$C$}\label{sec:prepC}

We would like to divide~$C$ into a number of paths and use
Lemma~\ref{lemma:embed} to obtain an~$\epsilon$-balanced
assignment of~$C$ to~$R$. Since we have split our graph~$G$
into two parts, we have to split~$C$ into two paths~$P_A$ and
$P_B$ and embed these into (an oriented graph similar
to)~$G[A']$ and~$G_B$ respectively.

Define $r:=4\lceil \log_2(4/\alpha) \rceil$.
Lemma~\ref{lemma:exact_length} tells us that if~$P$ is an
orientation of a path of length~$r$ then between any two
distinct vertices in~$V(R_B)$ or in~$V(R_A)$ there exists a
path isomorphic to~$P$.

Define
\[
s:=\lfloor(\log n)^2\rfloor,\quad t:=\left\lfloor\frac{n-(s+1)(r-1)}{s+2}\right\rfloor - 1 \approx \frac{n}{(\log n)^2}.
\]
Recall that~$\Q$ is a maximal collection of neutral pairs
in~$C$ all at a distance of at least~3 from each other. If~$\Q$
is large, \ie we are in the case where~$C$ is far from~$C^*$,
let~$v^*$ be a vertex in~$C$ such that the subpath of~$C$ of
length~$n/2$ following~$v^*$ and the subpath of~$C$
preceding~$v^*$ both contain at least~$2|\Q|/5$ elements
of~$\Q$. Divide~$C$ into overlapping paths (by which we mean
paths sharing endvertices)
\[
C:=Q_1P_1Q_2P_2\ldots Q_{s-1}P_{s-1}Q_sP_sQ^*P^*
\]
where their lengths satisfy~$\ell(P_i)=t$,
$\ell(Q_i)=\ell(Q^*)=r$ and $2t\leq\ell(P^*)<3t$ and~$Q_1$
starts at~$v^*$. Let~$s_B\in \N$ be such that
\[
1<n_B-s_B(t+r)<\ell(P^*)
\]
\COMMENT{We can do this as~$t+r+2<2t\leq \ell(P^*)$.}and let
\[
P_B:=P_B^*Q_1P_1\ldots Q_{s_B}P_{s_B}
\]
where~$P_B^*$ is an end-segment of~$P^*$ of such length as
to ensure~$\ell(P_B)+1=n_B$. Let
\[
P_A:=Q'_1P'_1\ldots Q'_{s_A}P'_{s_A}Q^*P_A^*
\]
where~$Q_i':=Q_{s_B+i}$, $P_i':=P_{s_B+i}$, $s_A:=s-s_B$
and~$P_A^*$ is an initial-segment of~$P^*$ which
overlaps~$P_B^*$ in exactly one place. Observe that we now have
\begin{equation}\label{eq:n_B}
n_B=s_Bt+s_Br+\ell(P_B^*)+1=|V(P_B)|
\end{equation}
and define
\[
n_A:=n-n_B=s_At+(s_A+1)r+\ell(P_A^*)+1 = |V(P_A)|-2.
\]

\section{Cycle is Far From~$C^*$}\label{sec:far}

\subsection{Approximate Embedding}

First we use the probabilistic tools in
Section~\ref{sec:approx} to assign the paths~$P_i$ to the
clusters of~$R_B$ in such a way as to ensure that all the
clusters are assigned approximately the same number of vertices
and the neutral pairs are relatively evenly distributed.
Let~$\Q_B\subset\Q$ consist of all neutral pairs from~$\Q$
which are contained in the~$P_i$ and moreover are at a distance
of at least three from the ends of the~$P_i$. Apply
Lemma~\ref{lemma:embed} to~$R_B$,
$\P_B:=\{P_1,P_2\ldots,P_{s_B}\}$ and~$\Q_B$ with~$\epsilon^*$
as~$\gamma$ to obtain an embedding of the~$P_i$ into~$V(R_B)$
with
\[
\left|a(i)-\frac{s_Bt}{M_B}\right| \leq \epsilon^* s_Bt, \quad
\left|n(i,\Q_B)-\frac{|\Q_B|)}{M_B}\right| \leq \epsilon^* s_B t.
\]
for all~$V_i\in V(R_B)$. (Recall that~$a(i)$ is defined to be
the number of vertices of $P_B$ assigned to the cluster~$V_i$
by the embedding.) In a slight abuse of notation let~$n(i)$ be
the number of neutral pairs in~$\Q_B$ starting at~$V_i$. Note
that
\begin{equation}\label{eq:a_im_B}
\left|a(i)-m_B\right|
    \stackrel{\eqref{eq:n_B}}{\leq} \left|a(i)-\frac{s_Bt}{M_B}\right| + \left|\frac{s_Br+3t}{M_B}\right|
    \leq \left|a(i)-\frac{s_Bt}{M_B}\right| + \epsilon^* m_B.
\end{equation}

The requirement that the neutral pairs in~$\Q$ are at a
distance of at least three from each other means that~$|\Q|\geq
n(C)/4$. By the observation in Section~\ref{sec:prepC} we know
that~$P_B$ contains at least~$2|\Q|/5 \geq \lambda n/10$
neutral pairs. The paths~$Q_i$ and~$P_B^*$ together contain
fewer than~$s_Br+3t$ neutral pairs and at most~$4s_B$ neutral
pairs can be in the~$P_i$ but within a distance of at most
three from a~$Q_i$. Thus for all~$i$
\begin{equation*}
n(i) \geq \frac{\lambda n}{10M_B} - \epsilon^* s_Bt - (s_Br+3t+4s_B)
\geq \frac{\lambda n_B}{6M_B} -2\epsilon^* n_B
\geq \frac{\lambda m_B}{7}.
\end{equation*}
For all~$2\leq i\leq s_B$ we can join~$P_{i-1}$ and~$P_i$ by a
path in~$R_B$ isomorphic to~$Q_i$ using
Lemma~\ref{lemma:exact_length}. Furthermore we can greedily
extend~$P_1$ backwards by a path isomorphic to~$P_B^*Q_1$. This
will increase~$a(i)$ by at most~$s_Br+3t < \epsilon^* m_B$
for~$n$ sufficiently large. We now have an assignment of~$P_B$
to the clusters of~$R_B$ which we can think of as a walk~$W_B$
in~$R_B$.

\subsection{Incorporating the Exceptional
Vertices}\label{subsec:far_except}

Let~$G_B^c$ be the digraph obtained from the pure oriented
graph~$G_B^*$ by making all the non-empty bipartite subgraphs
between the clusters complete (and orienting all the edges
between these clusters in the direction induced by~$R_B$) and
adding the vertices in~$V_0'$ as well as all the edges of~$G$
between~$V_0'$ and $V(G_B-V_0')$. Our next aim is to
incorporate the exceptional vertices~$V_0'$ into the
walk~$W_B$. We do this by considering the following extension
of~$R_B$. Define~$R_B^*\supseteq R_B$ to be the digraph formed
by adding to~$R_B$ the vertices in~$V_0'$ and, for $v\in V_0'$
and $V_i\in V(R_B)$, the edge $vV_i$ if $|N^+_G(v)\cap V_i|
> \alpha m_B/10$ and $V_iv$ if  $|N^-_G(v)\cap V_i|
> \alpha m_B/10$.

\begin{figure}
  \centering\footnotesize
  \includegraphics[scale=1]{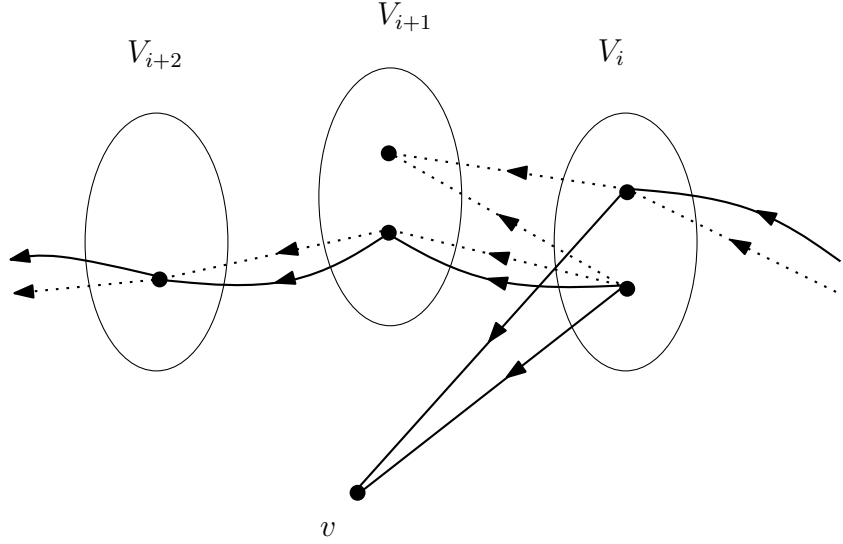}
  \caption{Incorporating an exceptional vertex when~$C$ is far from~$C^*$.}\label{fig:with_clusters}
\end{figure}

Then for each~$v\in V_0'$ pick an inneighbour $V_i\in V(R_B)$
and change the assignment of one neutral pair currently mapped
to~$V_iV_{i+1}V_i$ to~$V_ivV_i$. We can always find such an
inneighbour as~\eqref{eq:d+split} implies that each exceptional
vertex sees at least a three-eighths proportion of the vertices
in~$V(G_B)$. This reduces~$a(i+1)$ and~$n(i)$ by one.
Figure~\ref{fig:with_clusters} contains an illustration of
this, where we consider~$W_B$ as being in~$G_B^c$ and the
dotted lines as the section of the embedding to be replaced by
the solid lines. After doing this for every exceptional vertex
we will have that for all~$V_i\in V(R_B)$
\begin{equation}\label{eq:a_ibound}\begin{split}
\left|a(i)-m_B\right| &\stackrel{\eqref{eq:a_im_B}}{\leq}
    \left|a(i)-\frac{s_Bt}{M_B}\right| + \epsilon^* m_B \\
&\leq \left(\epsilon^* s_Bt+\epsilon_A m_B+|V_0'|\right) + \epsilon^* m_B
\stackrel{\eqref{eq:V'0}}{<} 4\epsilon_A n_B,
\end{split}\end{equation}
where the second term in the second line comes from greedily
embedding the~$Q_i$. We also still have a reasonable number of
neutral pairs starting at each cluster of~$R_B$ for all~$V_i\in
V(R_B)$:
\begin{equation*}
n(i) \geq \frac{\lambda m_B}{7} -  |V_0'| > \frac{\lambda m_B}{7} - \epsilon_A n_B > \frac{\lambda m_B}{8}.
\end{equation*}
Note that of the~$a(i)$ vertices of~$P_B$ assigned to
any~$V_i\in V(R)$, at most~$\epsilon_A n_B + 2|V_0'| \leq
3\epsilon_A n_B$ do not have their neighbours assigned
to~$V_{i-1}\cup V_{i+1}$, where the first term came from
the~$Q_i$ and the second came from incorporating the
exceptional vertices. Thus we currently have a
$(4\epsilon_A,3\epsilon_A)$-corresponding embedding of~$P_B$
into~$R_B^*$.

\subsection{Adjusting the Embedding}

We now adjust~$W_B$ to obtain a~$5\epsilon_AM_B$-corresponding
assignment of~$P_B$ to~$R_B^*$; i.e. we adjust~$W_B$ to ensure
that~$a(i)=m_B$ for all~$V_i\in V(R_B)$. Recall from
Corollary~\ref{cor:join} that between any two vertices in~$R_B$
there exists a skewed traverse of length at most~$r':=\lceil
4/\alpha \rceil$. Then for each cluster~$V_i\in V(R_B)$
with~$a(i+1)>m_B$ pick~$V_j\in V(R_B)$ with~$a(j)<m_B$ and find
a skewed~$V_i$-$V_j$ traverse of length $q\leq
r'$:\COMMENT{Such a~$j$ must always exist as $\sum
a(i)=m_BM_B$.}
\[
V_iV_{k_1},V_{k_1-1}V_{k_2},V_{k_2-1}V_{k_3},\ldots
,V_{k_q}V_{k_q-1},V_{k_q-1}V_j.
\]
As discussed in Section~\ref{sec:ss} we can use this skewed
traverse to modify~$W_B$ to reduce~$a(i+1)$ by one,
increase~$a(j)$ by one and leave the number of vertices
assigned to every other cluster of~$R_B$ the same. We do this
by, for every~$0\leq p\leq q$, replacing a neutral
pair~$V_{k_p-1}V_{k_{p}}V_{k_p-1}$ in~$W_B$
by~$V_{k_p-1}V_{k_{p+1}}V_{k_p-1}$ where we
define~$V_{k_0-1}:=V_i$ and~$V_{k_{q+1}}:=V_j$.

Since $\sum_{i=1}^{M_B}|a(i)-m_B|\leq 4\epsilon_A M_B n_B$,
doing this will consume at most~$4\epsilon_A M_B n_B$ neutral
pairs starting at any vertex of~$R_B$. This is fine though as
for all~$V_i\in V(R_B)$ we have~$n(i)\geq \lambda m_B/8 \gg
4\epsilon_A M_B n_B$. Each cluster~$V_i$ now has at
most~$3\epsilon_An_B+4\epsilon_A M_B n_B<5\epsilon_A M_B n_B$
vertices of~$P_B$ assigned to it that do not have both their
neighbours assigned to~$V_{i-1}\cup V_{i+1}$. Hence we have
constructed a~$5\epsilon_AM_B$-corresponding embedding~$W_B$
of~$P_B$ into~$R_B^*$.

\subsection{Finding a copy of~$P_B$
in~$G_B$} \label{subsec:CBinGBfar}

We will now use Lemma~\ref{lemma:blowupused} to find a copy
of~$P_B$ in~$G_B$. To do this we use~$W_B$ to find an
embedding~$W_B'$ of~$P_B$ into~$G_B$ such that
\begin{itemize}
\item Every vertex of~$W_B$ in~$V_0'$ is unchanged in~$W_B'$.
\item Each appearance of a cluster of~$R_B$ in~$W_B$ is
    replaced by a unique vertex in the corresponding
    cluster in~$G_B$.
\item Every edge of~$W_B$ which does not lie upon an edge
    of~$F_B$ is mapped to an edge of~$G_B$.
\end{itemize}
First we split~$W_B$ into two digraphs~$W_B^1$ and~$W_B^2$.
Let~$W_B^1$ consist of all maximal walks
\[u_{i,1}u_{i,2}\ldots u_{i,\ell_i}\]
in~$W_B$ of length at least three whose edges all lie on~$F_B$.
Let~$W_B^2$ consist of everything not in~$W_B^1$. Then~$W_B^2$
is a union of walks~$v_{i,1}v_{i,2}\ldots v_{i,k_i}$, where we
relabel if necessary to ensure that~$u_{i,1}=v_{i-1,k_{i-1}}$
and~$u_{i,\ell_i}=v_{i,1}$. In the next paragraph we will
greedily find an embedding of~$W_B^2$ into~$G_B$ which will
satisfy the third requirement above.

The walks in~$W_B^2$ are of one of three types. The first type
comes from the incorporation of an exceptional vertex, in which
we have an exceptional vertex~$x\in V_0'$ and a cluster~$V_i\in
V(R_B)$ with $|N^-_G(x)\cap V_i|
> \alpha m_B/10$. In this case we choose any two distinct vertices~$u,v\in N^-_G(x)\cap V_i$,
which we can do as there are at most~$|V_0'|\ll \epsilon m_B
\ll \alpha m_B/10$ exceptional vertices. The second type comes
from the paths~$Q_i$ and the path~$P_B^*$. These we find
in~$G_B^*$ (and hence in~$G_B\supseteq G_B^*$) greedily. We can
do so as the total length of the~$Q_i$ is at most $s_B r +2t
\ll \epsilon m_B$ and all their edges are assigned to edges
in~$R_B$ corresponding to~$\epsilon$-regular pairs of density
at least~$d$ in~$G_B^*$. The final type are pairs of
edges~$ij,ji$ with~$i,j\in V(R_B)$ which come from the skewed
traverses used to ensure that the correct number of vertices
of~$P_B$ were assigned to each vertex of~$R_B$. There are at
most~$5\epsilon_A M_Bn\ll \epsilon m_B$ of these and so we can
again find these greedily.\COMMENT{Take a maximal matching, has
size at least~$(1-\epsilon)m_B$. Then done
as~$\epsilon/\alpha\ll 1$.} Note that our requirement that all
the neutral pairs in~$\Q$ are at a distance of at least three
from each other and the ends of the~$P_i$ implies that we have
now considered all possible walks in~$W_B^2$. To satisfy the
second condition above we simply assign each vertex of~$W_B$
not already assigned to a (distinct) vertex in the
corresponding cluster in~$G_B$. As~$W_B$ is balanced (i.e.\@
$W_B$ assigns exactly~$m_B$ vertices to each cluster) we can do
this.

For all~$i$ let~$S_i$ consist of the vertices of~$G_B-V_0'$ to
which the vertices of~$W_B^1$ that are not at the end of a path
have been assigned. We can now apply
Lemma~\ref{lemma:blowupused} to~$G_B-V_0'$ with~$W_B^1$ as~$H$,
the~$u_{i,1}$ and~$u_{i,\ell_i}$ as the~$x_P$ and~$y_P$
respectively and the~$S_i$ as just defined. Combining this with
the embedding of~$W_B^2$ into~$G$ gives us a copy of~$P_B$
in~$G_B$.

\subsection{Finding a copy of~$C$ in~$G$}\label{subsec:CinGfar}

Recalling how we `chopped up'~$C$ at the start of this section,
let~$u,v\in V(G_B)$ be the vertices to which the endpoints
of~$P_B$ were assigned. To complete the proof of this case we
find a copy of~$P_A$ in~$G_A:=G[A'\cup\{u,v\}]$ starting at~$v$
and ending at~$u$. We find a copy of~$P_A$ exactly as we found
the copy of~$P_B$ with three differences. Firstly there are no
exceptional vertices. Secondly, recalling that
\[
P_A:=Q'_1P'_1\ldots Q'_{s_A}P'_{s_A}Q^*P_A^*,
\]
we require that  the embeddings of~$Q_1'$ and~$P_A^*$ start and
end at~$v$ and~$u$ respectively. Since~$Q_1'$ is long enough
for Lemma~\ref{lemma:exact_length} we can specify the cluster
to which its initial vertex is assigned and use
Lemma~\ref{lemma:exact_length} to join it to~$P'_1$. We
embed~$P_A^*$ greedily and use~$Q^*$ and
Lemma~\ref{lemma:exact_length} to connect it with the rest of
the embedding. Hence we can indeed start and end at the
required vertices. This doesn't affect the constants in the
rest of the proof.\COMMENT{Beyond $+2\ll m_B$ possibly.} Since
the number of exceptional vertices and the imbalances created
by the approximate embedding are both small (and small as
functions of~$M_A$) we can proceed exactly as before and find
the desired cycle~$C$ in~$G$. The calculations work as before
as a result of us only having two exceptional vertices. The
equation~\eqref{eq:a_ibound} becomes
\begin{align*}
\left|a(i)-m_A\right| &\leq
    \left|a(i)-\frac{s_At}{M_A}\right| + \epsilon^* m_A \\
&\leq \left(\epsilon^* s_At+\epsilon_A m_A+|\{u,v\}|\right) + \epsilon^* m_A
\leq 4\epsilon_A m_A.
\end{align*}
Hence from Section~\ref{subsec:CBinGBfar} we now have
\[
\sum_{i=1}^{M_A}|a(i)-m_A|\leq 4\epsilon_A M_A m_A,
\]
which is fine as we will have that~$n(i)\geq \lambda m_A/8 \ll
4\epsilon_A M_A m_A$ for all clusters~$V_i'\in V(R_A)$.

\section{Cycle is Close to~$C^*$}\label{sec:close}

Our argument closely follows that in the previous section, the
difference being in the means of correcting imbalances. To
correct imbalances we will need long sections of~$P_B$ with no
changes in orientation.
Define~$\ell_B:=\lceil\frac{4}{\alpha}\rceil M_B$, which is at
least the maximum length of a shifted walk between two vertices
in~$R_B$. As before we split up~$C$ into~$P_A$ and~$P_B$, the
only difference being that we do not need a special
vertex~$v^*$ this time. Let~$\Q_B'$ consist be the largest
possible collection of paths in~$P_B$ of length~$3\ell_B$ all
at a distance of at least 3 from each other, oriented in the
same direction and containing no changes in orientation. We
will call these \emph{long runs}. There are at least
\[
m(P_B,\Q_B')\geq \frac{n_B}{3\ell_B+6} -2\lambda n \geq \frac{\alpha n_B}{14M_B}
\]
of these in~$P_B$. (We subtract~$2\lambda n$ not~$\lambda n$ as
both neutral pairs~$V_iV_{i+1}V_i$ and their
inverse~$V_iV_{i-1}V_i$ kill possible long runs.)

Let~$\Q_B$ be the subset of~$\Q_B'$ containing those long runs
contained in the~$P_i$, at a distance of at least 4 from the
ends of all the~$P_i$ and all oriented in the same direction.
We assume that these are all oriented in the same direction
as~$F_B$. Keeping only long runs oriented in one direction
loses us at most half of them. The paths~$Q_i$, the
path~$Q^*P_B^*$ and the 3 vertices neighbouring them in
the~$P_i$ in each direction can intersect at most~$2s+2$ of the
long runs and so, abusing notation slightly,
\[
m(\P_B) \geq \frac{\alpha n_B}{28M_B}-2s-2 \geq \frac{\alpha n_B}{30M_B}
\]
for sufficiently large~$n$, where we recall
that~$\P_B:=\{P_1,P_2\ldots,P_{s_B}\}$. Similarly
defining~$\ell_A:=\lceil\frac{4}{\alpha}\rceil M_A$ and~$\Q_A'$
and~$\Q_A$ in the obvious way we have $m(\P_A) \geq \alpha
(n_A)/30M_A$.

Apply Lemma~\ref{lemma:embed} to~$R_B$, $\Q_B$ and~$\P_B$
with~$\eps^*$ as~$\gamma$ to obtain an embedding of the~$P_i$
into~$V(R_B)$ with
\begin{equation}\label{eq:near_approx}
\left|a(i)-\frac{s_Bt}{M_B}\right| \leq \epsilon^* s_Bt, \quad
m(i) \geq \frac{\alpha n_B}{30M_B^2} - \epsilon^* s_Bt \geq \frac{\alpha n_B}{32M_B^2}
\end{equation}
for all~$V_i\in V(R_B)$, where we write~$m(i)$ for the number
of elements of~$\Q_B$ whose initial vertex is assigned
to~$V_i\in V(R)$.

For all~$2\leq i\leq s_B$ we can join~$P_{i-1}$ and~$P_i$ by a
path in~$R_B$ isomorphic to~$Q_i$ using
Lemma~\ref{lemma:exact_length}. Furthermore we can greedily
extend~$P_1$ backwards by a path isomorphic to~$P_B^*Q_1$. This
will increase~$a(i)$ by at most~$s_Br+2t\leq \epsilon_A m_B$
for~$n$ sufficiently large. We now have an embedding of~$P_B$
into~$R_B$ which we can think of as a walk~$W_B$ in~$R_B$.

\begin{figure}
  \centering\footnotesize
  \includegraphics[scale=0.8]{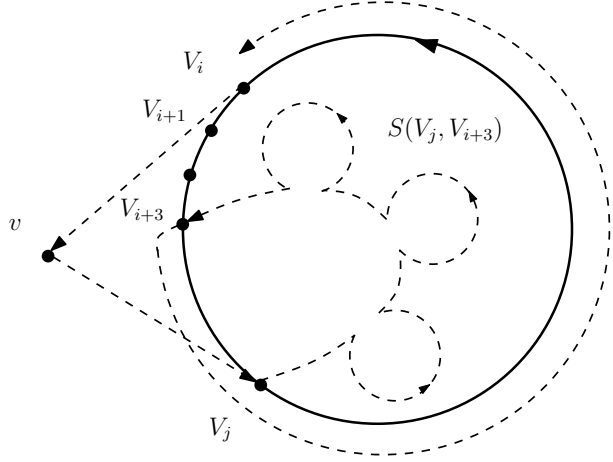}
  \caption{Incorporating an exceptional vertex when~$C$ is close to~$C^*$.}\label{fig:exceptional}
\end{figure}

Let~$G_B^*$, $G_B^c$ and~$R_B^*$ be defined exactly as in
Section~\ref{subsec:far_except}. Let~$v\in V_0'$ be an
exceptional vertex and let~$V_iv,vV_j\in E(R_B^*)$. ($V_i$
and~$V_j$ exist by~\eqref{eq:d+split}.) Take a long run
in~$\Q_B$ whose initial vertex is currently assigned to~$V_i$.
Since~$M_B$ divides~$\ell_B$ it also ends at~$V_i$. We cannot
replace the long run simply by~$V_ivV_jF_B\ldots F_B$ because
this would not end at~$V_i$. Thus it would require us to alter
the rest of our approximate embedding, possibly
causing~\eqref{eq:near_approx} to no longer hold. Instead we
use shifted walks and a `jump' to ensure that our modification
incorporates~$v$ into our walk and does not alter~$a(i)$
or~$m(i)$ significantly for any cluster of~$R_B$. We replace
the long run starting at~$V_i$ with the following walk
\[
V_ivV_jS(V_j,V_{i+3})F_BF_B\ldots F_BV_{i},
\]
where~$S(V_j,V_{i+3})$ is a shifted walk from~$V_j$
to~$V_{i+3}$. The number of~$F_B$ is chosen so that the new
section has exactly the same length as the long run it
replaces. This is illustrated in Figure~\ref{fig:exceptional}.
This is a walk that goes out to~$v$, back to~$V_j$, follows a
shifted walk to~$V_{i+3}$ and then winds around~$F$ until we
have a walk of length~$3\ell_B$ ending at~$V_i$. This new walk
visits~$V_{i+1}$ and~$V_{i+2}$ one time fewer than previously
and~$V_j$ one time more. Observe that the shifted walk by
definition visits every cluster in~$R_B$ the same number of
times, which allows us to observe that we still end at~$V_i$.
Repeating this for each exceptional vertex creates a new
assignment now satisfying
\begin{align*}
\left|a(i)-m_B\right| &\stackrel{\eqref{eq:n_B}}{\leq}
    \left|a(i)-\frac{s_Bt}{M_B}\right| + \left|\frac{s_Br+2t}{M_B}\right| \\
&\leq \left(\epsilon_A s_Bt+\epsilon_A m_B+|V_0'|\right) + \epsilon_A m_B
\leq 3\epsilon_A n_B.
\end{align*}
for all~$i$. We also still have a reasonable number of long
runs starting at each cluster.
\begin{align*}
m(i) &\geq \frac{\alpha n_B}{32M_B^2} - |V_0'| \geq \frac{\alpha n_B}{40M_B^2}.
\end{align*}
Note that of the~$a(i)$ vertices of~$P_B$ assigned to~$V_i\in
V(R)$, at most
\[\epsilon_A m_B + 4|V_0'| \leq 5\epsilon_A n_B \]
do not have their neighbours assigned to~$V_{i-1}\cup V_{i+1}$.
The first term here comes from connecting the~$P_i$ and the
second term from incorporating the exceptional vertices: each
exceptional vertex has one direct edge to or from a given
cluster in~$R_B$ and the shifted walk can add at most two edges
outside~$F_B$ to each cluster. Thus we currently have a
$(5\epsilon_A,3\epsilon_A)$-corresponding assignment of~$P_B$
into~$R_B^*$.

\subsection{Correcting the imbalances}\label{subsec:close_adjust}
We now adjust our current assignment of~$P_B$ to~$R_B^*$ to
obtain a~$15\epsilon_A$-corresponding assignment, i.e. we
adjust~$W_B$ to ensure that~$a(i)=m_B$ for all~$V_i\in V(R_B)$.
To do this we find a pair~$V_i,V_j\in V(R_B)$ such
that~$a(i)>m_B$ and~$a(j)<m_B$ and replace a long run starting
at~$V_{i-1}$ with the following walk:
\[ S(V_{i-1},V_j)S(V_j,V_{i+1})F_B\ldots F_BV_{i-1},\]
where the number of~$F_B$ is chosen to ensure that the new
section has length~$3\ell_B$. This walk removes the assignment
of one vertex to~$V_i$, assigns one extra vertex to~$V_j$ and
does not change the number of vertices assigned to all other
clusters in~$R_B$. Since~$\sum_{i=1}^{M_B}a(i)=m_BM_B$ we can
always find such a pair unless we have corrected all the
imbalances. Each pair requires a long run and we still have at
least~$\alpha n_B/40M_B^2 \gg 3\epsilon_A n_B$ of these
starting at each cluster and so can indeed correct all the
imbalances. This leaves us with a balanced assignment with at
most
\[
3\epsilon_A n_B + 4\cdot 3 \epsilon_A n_B = 15\epsilon_A n_B
\]
edges outside~$F_B$ from each vertex. Hence there are at
most~$15\epsilon_A M_Bn_B \ll \epsilon m_B$ edges in total not
in a path of length at least~3 all of whose edges lie on~$F_B$
or not lying entirely on~$F_B$. This is exactly the same
position as in Section~\ref{subsec:CBinGBfar}. We can now
proceed as before to first find a copy of~$P_B$ in~$G_B$ and
then repeat the procedure with~$P_A$ (using~$\Q_A$ not~$\Q_B$)
to find the desired cycle~$C$ in~$G$. This completes this
section and the proof of Theorem~\ref{thm:arb_ham}.

\section*{Acknowledgements}

The author would like to thank Oliver Cooley, Daniela K\"uhn
and Deryk Osthus for their assistance in the creation of this
paper.

\medskip

{\footnotesize \obeylines \parindent=0pt

Luke Kelly
School of Mathematics
University of Birmingham
Edgbaston
Birmingham
B15 2TT
UK
}

{\footnotesize \parindent=0pt

\it{E-mail address}:
\texttt{kellyl@maths.bham.ac.uk}
\end{document}